\documentclass[11pt]{amsart}
\usepackage{graphicx}
\usepackage{tgpagella}
\usepackage{euler}
\usepackage[T1]{fontenc}
\usepackage{amssymb}
\usepackage{amsmath,amsthm}
\usepackage[hidelinks]{hyperref}
\usepackage{amsthm}
\usepackage[english]{babel}
\usepackage{mathrsfs}

\newtheorem{main}{Theorem}

\newtheorem{theorem}{Theorem}[section]
\newtheorem{lem}[theorem]{Lemma}
\newtheorem{prop}[theorem]{Proposition}
\newtheorem{Conjecture}[theorem]{Conjecture}

\theoremstyle{definition}
\newtheorem{definition}[theorem]{Definition}
\newtheorem{notation}[theorem]{Notation}
\newtheorem{Remark}[theorem]{Remark}
\newtheorem{question}[theorem]{Question}
\def\ca{\curvearrowright}

\def\e{\epsilon}

\def\al{\alpha}
\def\bb{\mathbb}

\def\up{\upsilon}

\def\del{\partial}

\def\f{\flat}
\def\Cal{\mathcal}

\def\inv{{\scriptscriptstyle {-1}}}
\def\HFD{H$_{\text{\tiny FD}}$}

\DeclareMathOperator*{\GL}{{GL}}

\DeclareMathOperator*{\diam}{{diam}}

\numberwithin{equation}{section}
\newcommand{\ip}[2]{\langle #1, #2 \rangle}
\newcommand{\abs}[1]{\lvert#1\rvert}
\providecommand{\bbs}[1]{\Bigl\lvert #1 \Bigr\rvert}
\providecommand{\nor}[1]{\lVert #1 \rVert}
\providecommand{\bnor}[1]{\Bigl\lVert #1 \Bigr\rVert}

\begin{document}
\title[Affine Ergodic Theorems]{On the Ergodic Theorem for Affine Actions\\ on Hilbert Space}
\author[I. Chifan]{Ionut Chifan}
\address{Department of Mathematics, University of Iowa, 14 MacLean Hall, IA
52242, USA and IMAR, Bucharest, Romania}
\email{ionut-chifan@uiowa.edu}
\thanks{I.C.\ was supported in part by NSF Grant \#1263982.}
\author[T. Sinclair]{Thomas Sinclair}
\address{Department of Mathematics, University of California, Los Angeles,
Box 951555, Los Angeles, CA 90095-1555, USA}
\email{thomas.sinclair@math.ucla.edu}
\thanks{T.S.\ was supported by an RTG Assistant Adjunct Professorship}
\date{\today }
\dedicatory{}

\begin{abstract}
The note establishes a new weak mean ergodic theorem (Theorem \ref%
{main:norm}) for $1$-cocycles associated to weakly mixing representations of
amenable groups. 
\end{abstract}

\maketitle

\section*{Introduction}

In a groundbreaking paper \cite{Shalom}, Shalom discovered several deep
connections between the representation theory of an amenable group and
aspects of its large-scale geometry. One motivation for his work, among
several others, was the development of a ``spectral'' approach to Gromov's
celebrated theorem on the virtual nilpotency of groups of polynomial growth 
\cite{Gromov}. More precisely, Shalom established, Theorem 1.11 in \cite%
{Shalom}, that if it could be shown that any group of polynomial growth $G$
possessed property H$_{\text{{\tiny FD}}}$ (see Definition \ref{h-fd}), then this would suffice to
establish that $G$ would have a finite-index subgroup with infinite
abelianization---the key step in Gromov's proof which involves the use of
Hilbert's 5$^{\text{th}}$ problem. As a means of establishing property H$_{\text{{\tiny FD}}}$, Shalom conjectured that for a group of
polynomial growth, a sequence of almost fixed points for any affine action
with weakly mixing linear part could be obtained by averaging the associated 
$1$-cocycle over an appropriate subsequence of $n$-balls centered at the
identity: see section 6.7 in \cite{Shalom}. In their paper \cite{CTV},
Cornulier, Tessera, and Valette made decisive contributions to Shalom's
program. In particular, their investigation of averaging properties of
groups over controlled F\o lner sequences has directly influenced the
approach taken in this paper.

\subsection*{Statement of results}

We establish the following weak mean ergodic theorem for affine actions of finitely-generated amenable groups on Hilbert space.
We will say that a sequence $(\mu_n)$ of regular Borel probability measures
on a countable, discrete group $G$ forms a \emph{Reiter sequence} if $%
\nor{\mu_n - g\ast\mu_n}\to 0$ for all $g\in G$, where $g\ast\mu_n(h) =
\mu_n(g^{{\scriptscriptstyle {-1}}}h)$. A countable discrete group is said
to \emph{amenable} if it admits a Reiter sequence.

\begin{main}[Weak Mean Ergodic Theorem]
\label{main:norm} Let $\pi: G\to \mathscr O(\mathcal{H})$ be an ergodic
orthogonal representation of a finitely generated amenable group $G$, and
let $b: G\to \mathcal{H}$ be a $1$-cocycle associated to $\pi$. Let $S$ be a
finite symmetric generating set for $G$, and let $\lvert \,\cdot\,\rvert$
denote the word length in $S$. If $(\mu_n)$ is a Reiter sequence for $G$,
then 
\begin{equation}  \label{wmet}
\int \frac{1}{\lvert g\rvert}\,b(g)\, d\mu_n(g)\to 0
\end{equation}
in the weak topology on $\Cal H$. If $\pi$ is weakly mixing, then 
\begin{equation}  \label{wmwmet}
\int \frac{1}{\lvert g\rvert}\,\lvert \langle b(g), \xi \rangle\rvert\,
d\mu_n(g)\to 0
\end{equation}
for all $\xi\in \mathcal{H}$.
\end{main}

\noindent We note that while $\frac{1}{\abs{e}}$ is technically undefined, by convention it will be understood to denote $0$ here and throughout.

In particular, if $b: \mathbb{Z}\to \mathcal{H}$ is a $1$-cocycle, then $b$
is completely determined by $\xi\doteq b(1)$, so that for $n\geq 1$ we have $%
\frac{1}{n}\, b(n) = A_n(\xi) := \frac{1}{n}\sum_{k=0}^{n-1}\pi(k)\xi$:
a similar formula holds for $-n$ via the identity $b(-n) = -\pi(-n)b(n)$.
So, in this case the result reduces to the fact that the Ces\`aro sums $%
C_n(\xi,\eta) = \frac{1}{n}\sum_{k=1}^n \langle A_k(\xi), \eta \rangle$ and $%
C_n^{\prime }(\xi,\eta) = \frac{1}{n}\sum_{k=1}^n \lvert \langle A_k(\xi),
\eta \rangle\rvert$ converge to $0$ for all $\xi,\eta\in \mathcal{H}$. The stronger summation holds for all ergodic representations and is
equivalent to the (weak) mean ergodic theorem of von Neumann.

In fact, for the class of abelian groups, we will give a new ``geometric'' proof of the mean ergodic theorem as a consequence of the following theorem in combination with Theorem \ref{main:norm}.

\begin{main}
\label{main:controlled} Let $G$ be finitely generated amenable group
admitting a \emph{controlled F\o lner sequence} (see Definition \ref{defn:contr-Folner}). Let $\pi: G\to\mathscr O(\Cal H)$ be an orthogonal representation, and let $b: G\to \Cal H$ be a 1-cocycle associated to $\pi$. Suppose that \begin{equation}\int \frac{1}{\abs{g}}\ip{b(g^{\inv})}{\xi} d\mu_n(g)\to 0\end{equation} for all $\xi\in \Cal H$ and all Reiter sequences $(\mu_n)$. Then the affine action
$G\curvearrowright^T \mathcal{H}$ associated to $b$ admits a
sequence of almost fixed points.
\end{main}

To see how this implies the mean ergodic theorem for $\bb Z$, we point out that by an observation of Cornulier--Tessera--Valette (Proposition 3.1 in \cite{CTV}) a consequence of $G\ca^T\Cal H$ admitting almost fixed points is that \[\frac{1}{\abs{g}}\nor{b(g)}\to 0\] as $\abs{g}\to\infty$; in other words, the $1$-cocycle $b$ has \emph{sublinear growth}. In fact, sublinearity of a $1$-cocycle is actually equivalent in general to the mean ergodic theorem, i.e.,   
the statement that
\[\int \frac{1}{\abs{g}}\nor{b(g)} d\mu_n(g)\to 0\] for all Reiter sequences $(\mu_n)$ (see Proposition \ref{thm:h-rigid}).

The significance of averaging on the right rather than on the left in Theorem \ref{main:controlled} is the it allows one to conclude that the cocycle is weakly sublinear, i.e., $\frac{1}{\abs{g}} b(g)\to 0$ in the weak topology, from which point averaging arguments over a controlled F\o lner sequence produce the desired sequence of almost fixed points. In the weak mixing case, however, one already knows that the $1$-cocycle must be ``almost weakly sublinear'' in the sense that for any $\e>0$ and $\xi\in \Cal H$, the subset consisting of all elements $g\in G$ such that $\abs{\ip{b(g)}{\xi}}\geq \e\abs{g}$ has measure $0$ for all left invariant means on $G$. In the case of a \emph{compact} representation, this would be sufficient to conclude weak sublinearity. It seems plausible that under additional structural assumptions on the $1$-cocycle or the group, one may be able to derive weak sublinearity in the general case.

\subsection*{Remarks on the proofs}

The paper is an application of the authors' philosophy of investigating the
``large scale'' properties of affine actions of groups on Hilbert space.
Though the above results are stated for affine actions, even in this case
the proofs crucially rely on a coarsening of the notion of an affine
action, the concept of an \emph{array}---the terminology is meant to evoke the geometric regularity of the orbits of
affine actions---formalized by the authors in \cite{CS}. The main reason is
that starting from an affine action $G\curvearrowright^T\mathcal{H}$, from
the associated $1$-cocycle $b$, one can construct an array $\alpha:G\to 
\mathcal{V}$ into a $G$-invariant positive cone $\mathcal{V}\subset \mathcal{%
H}\otimes\mathcal{H}$. Such a map cannot lie within a uniformly bounded
distance of an unbounded $1$-cocycle, since the equation $b(g) = -\pi_g b(g^{%
\scriptscriptstyle {-1}})$ holds for all $g\in G$ for any $1$-cocycle $b$.
Some potential implications of the positivity of this map will be discussed in Section \ref{Final-Remarks}

The notion of an array is best viewed from a geometric, rather than
algebraic, perspective. Indeed, a length function on a discrete group $G$
may be viewed as a positive array associated with the trivial
representation. In general, an array can be thought of as a Hilbert-space
valued length function on $G$ which is compatible with some orthogonal $G$%
-representation $\pi$. The presence of an array then becomes a tool through
which properties of the representation can be used to impose large scale
conditions on the group, and \emph{vice versa}. For example, it is shown in 
\cite{CS}, Proposition 1.7.3, that a non-amenable group admitting a proper
array into its left-regular representation, e.g., non-elementary Gromov hyperbolic groups, cannot
be decomposed as a direct product of infinite groups. Turning to the topic
at hand, the presence of a controlled F\o lner sequence imposes a strong
large-scale ``finite dimensionality'' condition on the group $G$---for the
case of weak polynomial growth, a point already well made in \cite{Gromov}.
Viewed in this light, the content of Theorem \ref{main:controlled} is that
this forces any geometric realization of the group which is uniformly
distributed throughout an infinite-dimensional Hilbert space to be
essentially degenerate.

\section{Geometry and Representation Theory}

\label{sec:geom-repn} In this section we will introduce the main
definitions, concepts, and perspective to provide a context for the
exposition.

\begin{notation}
Let $X$ be a set and let $f,g: X\to \mathbb{R}_{\geq 0}$ be maps. We write $%
f = O(g)$ or $f\ll g$ if there exists a finite set $F\subset X$ and a
constant $C>0$ such that $f(x)\leq C\cdot g(x)$ for all $x\in X\setminus F$.
We will write $f \lesssim g$ if $f\ll g$ for a constant $C\leq 1$. We write $%
f = o(g)$ if $f\lesssim \epsilon\cdot g$ for all $\epsilon>0$.
\end{notation}

\subsection{Isometric actions on Hilbert space}

\begin{definition}
An orthogonal representation $\pi: G\to \mathscr O(\mathcal{H})$ is said to
be \emph{ergodic} if for any $\xi\in \mathcal{H}$ we have that $\pi_g(\xi) =
\xi$ for all $g\in G$ if and only if $\xi = 0$, i.e., $\pi$ has no non-zero
invariant vectors. The representation $\pi$ is said to be \emph{weakly mixing%
} if the diagonal representation $\pi\otimes\pi: G\to \mathscr O(\mathcal{H}%
\otimes \mathcal{H})$ is ergodic. In particular weakly mixing
representations are ergodic.
\end{definition}

If $\pi: G\to \mathscr O(\mathcal{H})$ is an orthogonal representation, a
map $b: G\to \mathcal{H}$ is said to be a $1$-cocycle associated to $\pi$ if
it satisfies the Leibniz identity 
\begin{equation*}
b(gh) = \pi_g(b(h)) + b(g),
\end{equation*}
for all $g,h\in G$. It is essentially a consequence of the Mazur--Ulam
theorem that any isometric action $G\curvearrowright^T \mathcal{H}$ may be
written as $T_g(\xi) = \pi_g(\xi) + b(g)$ for some orthogonal representation 
$\pi$ and an associated $1$-cocycle $b(g)$ and conversely. The
representation $\pi$ is known as the \emph{linear part} of $T$.

\begin{definition}
An isometric action $G\curvearrowright^T \mathcal{H}$ is said to admit \emph{%
almost fixed points} if there exists a sequence $(\xi_n)$ of vectors in $%
\mathcal{H}$ such that 
\begin{equation*}
\nor{T_g(\xi_n) - \xi_n}\to 0
\end{equation*}
for all $g\in G$.
\end{definition}

\begin{definition} We will say that a $1$-cocycle $b$ associated to an orthogonal representation $\pi: G\to \mathscr O(\Cal H)$ is \emph{almost inner} if the associated affine isometric action $G\ca \Cal H$ admits amost fixed points.
\end{definition}

\subsection{Geometric group theory}

\label{sub:poly-growth} Throughout the paper $G$ will be a countable
discrete group, often finitely generated. Recall that a \textit{length
function} $\lvert \,\cdot\,\rvert: G\to \mathbb{R}_{\geq 0}$ is a map
satisfying: (1) $\lvert g\rvert = 0$ if and only if $g = e$ is the identity;
(2) $\lvert g^{{\scriptscriptstyle {-1}}}\rvert = \lvert g\rvert$, for all $%
g\in G$; and (3) $\lvert gh\rvert \leq \lvert g\rvert + \lvert h\rvert$, for
all $g,h\in G$. A length function is proper if the map $g\mapsto \lvert
g\rvert$ is proper, i.e., all sets of bounded length are finite. If $\lvert
\,\cdot\,\rvert$ is a length function, then we denote 
\begin{equation*}
B(n) = \{g\in G : \lvert g\rvert\leq n\},
\end{equation*}
the \emph{ball of radius $n$} centered at the identity, and 
\begin{equation*}
S(n) = \{g\in G : \lvert g\rvert = n\},
\end{equation*}
the \emph{sphere of radius $n$} centered at the identity. If $G$ is
generated by a finite set $S$, then the function which assigns to each $g\in
G$ the least integer $k$ such that $g$ can be written as a product of $k$
elements from $S\cup S^{{\scriptscriptstyle {-1}}}$ is a proper length
function, known as a \emph{word length function}.

\begin{notation}
Let $G$ be a finitely generated, discrete group with a fixed finite,
symmetric, generating set $S$. Let $F\subset G$ be a finite subset. We set 
\begin{equation*}
\partial F := \bigcup_{g\in S} g F\Delta F,
\end{equation*}
where ``$\Delta$'' denotes the symmetric difference.
\end{notation}

\begin{definition}
\label{defn:Folner-seq} A sequence $(F_n)_{n\in\mathbb{N}}$ of finite
subsets of $G$ is said to form a \emph{F\o lner sequence} if 
\begin{equation*}
\frac{\lvert gF_n\Delta F_n\rvert}{\lvert F_n\rvert}\to 0
\end{equation*}
for all $g\in G$.
\end{definition}

\begin{definition}
\label{defn:contr-Folner} Let $G$ be a finitely generated, discrete group
with a fixed finite, symmetric, generating set $S$. For a constant $K>0$, a sequence $(F_n)_{n\in%
\mathbb{N}}$ of finite subsets of $G$ is said to be a \emph{$K$-controlled F\o %
lner sequence} if 
\begin{equation*}
\frac{\lvert \partial F_n\rvert}{\lvert F_n\rvert}\leq \frac{K}{\diam F_n}%
,
\end{equation*}
where $\diam F_n$ is defined to be the least integer $m$ such that $%
F_n\subset B(m)$. The group admits a controlled F\o lner sequence if it admits a $K$-controlled F\o lner sequence for some $K$.
\end{definition}

\begin{definition}
A finitely generated group $G$ is said to have \emph{polynomial growth} if
for some (equivalently, for any) proper word length function we have that 
\begin{equation*}
\limsup_n \frac{\log \lvert B(n)\rvert}{\log n}<\infty.
\end{equation*}

\noindent The group $G$ is said to be of \emph{weak polynomial growth} if 
\begin{equation*}
\liminf_n \frac{\log \lvert B(n)\rvert}{\log n}<\infty
\end{equation*}
for any proper word length.
\end{definition}

The following observation is due to Shalom.

\begin{prop}[Shalom, Lemma 6.7.3 in \protect\cite{Shalom}]
\label{prop:poly-cf} If $G$ is a finitely generated group of polynomial
growth of degree $d$, then for any proper word length, there is a
subsequence $\mathcal{S}\subset \mathbb{N}$ such that the sequence of balls $%
(B(n))_{n\in \mathcal{S}}$ form a $K$-controlled F\o lner sequence for $K>
10d$.
\end{prop}

In fact, a group $G$ which satisfies a doubling condition $\lvert
B(2n)\rvert\leq C\cdot \lvert B(n)\rvert$ for some subsequence admits a
controlled F\o lner sequence by an observation of Tessera, \cite{Tessera2},
Remark 4.10. Gromov's ``Regularity lemma'' (\cite{Gromov}, section 3) shows
that groups of weak polynomial growth have the doubling condition. By the
work of Tessera several large classes of groups are known to admit
controlled F\o lner sequences.

\begin{prop}[Tessera, Theorem 11 in \protect\cite{Tessera2}]\label{prop-cf}
The following classes of groups admit controlled F\o lner sequences:

\begin{enumerate}
\item polycyclic groups;

\item wreath products $D \wr \mathbb{Z}$ with $D$ finite;

\item semi-direct products $\mathbb{Z}[\frac{1}{mn}]\rtimes_{m/n} \mathbb{Z}$%
, with $m,n$ coprime and $\lvert mn\rvert\geq 2$.
\end{enumerate}
\end{prop}

By results of Mal'cev and Auslander, it is known that a group $G$ is
polycyclic if and only if $G$ is realizable as a solvable subgroup of $\GL(n,%
\mathbb{Z})$, cf.\ \cite{de-la-Harpe}, section III.A.5.

The full extent of the class of amenable groups admitting a
controlled F\o lner sequence is unknown.  An
interesting problem would be to determine exactly which solvable groups with
finite Hirsch number belong to this class or at least have property H$_{%
\text{{\tiny FD}}}$. (To recall, let $G$ be a solvable group with derived series $G > G^{(1)} > G^{(2)} >
\dotsb > G^{(n)} > G^{({n+1})} = \{1\}$. The \emph{Hirsch number} is then
defined to be the sum of the torsion-free ranks of the abelian groups $%
G^{(i)}/G^{(i+1)}$, $i = 1,\dots,n$. See section 6.6 in \cite{Shalom} for a
discussion on this problem.) We pose the following,
more concrete question:

\begin{question}
If $\Gamma$ is a solvable subgroup of $\GL(n,\mathbb{Z}[\frac{1}{p}])$, does 
$\Gamma$ admit a controlled F\o lner sequence?
\end{question}

\noindent We remark that $\mathbb{Z}[\frac{1}{p}]$ can probably not be replaced with $%
\mathbb{Z}[\tau]$ for some non-algebraic number $\tau$, since $\GL(2,\mathbb{%
Z}[\tau])$ contains a copy of $\mathbb{Z}\wr \mathbb{Z}$ which ought not to admit a controlled F\o lner sequence.

\subsection{Arrays}

The definition of an array was formally introduced in \cite{CS} as a means for unifying the concepts
of length functions and $1$-cocycles into orthogonal representations. Arrays
fit naturally into the theory of the large scale geometry of discrete
groups, being closely related to in particular Guoliang Yu's Property A,
cf.\ \cite{Roe}, and Ozawa's class $\mathcal{S}$, cf. \cite{BrOz}. (See section 1 of \cite{CS} for an in-depth
discussion of the relationship between arrays with values into the
left-regular representation and ``negative curvature'' in geometric group
theory.) We now recall the definition.

\begin{definition}
\label{def:array} Let $\pi: G\to \mathscr O( \mathcal{H})$ be an
orthogonal representation of a countable discrete group $G$. A map $\alpha:
G\to \mathcal{H}$ is called an \textit{array} if for every finite subset 
$F\subset G$ there exists $K\geq 0$ such that 
\begin{equation}  \label{weakquasi}
\nor{\pi_g(\al(h)) - \al(gh)}\leq K,
\end{equation}
for all $g\in F$, $h\in G$ (i.e., $\alpha$ is \textit{boundedly equivariant}%
). It is an easy exercise to show that for any array $\alpha$ on a finitely
generated group $G$ there exists a proper word length function on $G$, a
scalar multiple of which bounds $\nor{\al(g)}$ from above.
\end{definition}

\begin{lem}
\label{lem:tensor} Let $G$ be a finitely generated group equipped with some
proper word length associated to a finite, symmetric, generating set $S$. If 
$\alpha: G\to \mathscr O(\mathcal{H})$ is an array into an orthogonal
representation $\pi$, then $\widetilde\alpha(g) := \frac{1}{\lvert
g\rvert}\ \alpha(g)\otimes\alpha(g)$, with $\widetilde\alpha(e) := 0$,
is an array into $\pi\otimes\pi$.
\end{lem}

\begin{proof}
The proof is very similar to the proof of Proposition 1.4 of \cite{CS}: we
include it here only for the sake of completeness. First, for every $g\in G$%
, we denote by $B_g := \sup_{h\in G}\|\alpha(gh)-\pi_g(\alpha(h))\|$ and
from the assumptions we have $B_g<\infty$. Using the triangle inequality
together with the bounded equivariance property, for all $k\in G$ we have $%
\|\alpha(k)\|\leq D|k|$, where $D=\max_{s\in S}B_s$. This further implies
that for every $\ell\in G$ we have the following inequality 
\begin{equation}  \label{100123}
\begin{split}
\sup_{k\neq e,\ell^{{\scriptscriptstyle {-1}}}}\frac{\|\alpha(k)\|}{|\ell k|}
=\sup_{k\neq e,\ell^{{\scriptscriptstyle {-1}}}}\frac{\|\alpha(k)\|}{| k|}%
\cdot \frac{|k|}{|\ell k|} \leq D (|\ell|+1).
\end{split}%
\end{equation}

To check the bounded equivariance for $\widetilde\alpha$, we fix $g,h\in G$
where $h\neq e, g^{{\scriptscriptstyle {-1}}}$. Applying the triangle
inequality and using successively the bounded equivariance property, the
basic inequality $||gh|-|h|| \leq |g|$, and the inequality (\ref{100123}),
we have 
\begin{equation*}
\begin{split}
\|\widetilde\alpha(gh)-(\pi\otimes \pi)_g\widetilde\alpha(h)\|& \leq \frac{%
\|(\alpha(gh)-\pi_g\alpha(h))\otimes \alpha(gh)\|}{|gh|}+\frac{%
\|\pi_g\alpha(h)\otimes(\alpha(gh)-\pi_g\alpha(h))\|}{|gh|}+ \\
& \quad +\|\pi_g\alpha(h)\otimes\pi_g\alpha(h)\| \left |\frac {1}{|gh|}-%
\frac{1}{|h|}\right | \\
&\leq B_g \frac{\|\alpha(gh)\|}{|gh|}+B_g\frac{\|\alpha(h)\|}{|gh|}+
||gh|-|h|| \frac{\|\alpha(h)\|}{|h|}\frac{\|\alpha(h)\|}{|gh|} \\
& \leq B_gD(|g|+2)+ D^2|g| (|g|+1).
\end{split}%
\end{equation*}

\noindent This implies that for every $g,h \in G$ we have 
\begin{equation*}
\|\widetilde\alpha(gh)-(\pi\otimes \pi)_g\widetilde\alpha(h)\| \leq \max
\{B_gD(|g|+2)+ D^2|g| (|g|+1),\|\widetilde\alpha(g^{{\scriptscriptstyle {-1}}%
})\|, \|\widetilde\alpha(g)\|\},
\end{equation*}
which concludes our proof as the right hand expression depends only on $g$.
\end{proof}

\subsection{Large scale lipschitz maps}Let $\mathcal{V}$ be a normed vector space. We will say a map $f: G\to \mathcal{V}$
is \emph{large scale lipschitz} if there exists a map $C: G\to \mathbb{R}_{\geq 0}$ such
that for all $g\in G$, $\nor{f(g) - f(gs)}\leq C(s)$. An array can be viewed
in some sense as the formal ``adjoint'' of some large scale lipschitz map $f: G\to 
\mathcal{H}$ with respect to the representation $\pi$, viz.,

\begin{prop}
If $\alpha: G\to \mathcal{H}$ is an array associated to $\pi$, then $%
\alpha^\star(g) := \pi(g)\alpha(g^{\scriptscriptstyle {-1}})$ is large scale lipschitz. Conversely, if $f: G\to \mathcal{H}$ is large scale lipschitz, then $%
f^\star(g) := \pi(g)f(g^{\scriptscriptstyle {-1}})$ is an array
associated to $\pi$.
\end{prop}

\noindent The proof consists of a straightforward check that the respective identities are satisfied.

 Given a finite, symmetric generating set $S$ for $G$, for any map $f: G\to \bb R$ we define the variation function $\del f: G\to \bb R^{S}$ by $\del f(g)(s) := f(g) - f(gs)$.

\begin{definition}\label{defn:higson} A bounded function $f: G\to \bb R$ is said to be \emph{higson} if $\nor{\del f}\in C_0(G)$, where $\nor{\,\cdot\,}$ is the euclidean norm on $\bb R^S$.
\end{definition}

\noindent Note that if $f: G\to \Cal V$ is a large scale lipschitz map into a normed vector space $\Cal V$, then $g\mapsto \frac{1}{\abs{g}} f(g)$ is higson.

 We define
 $\mathscr H^\infty(G)$ to be Banach space of all higson functions. For all $1\leq p< \infty$, we also define $\mathscr H^p(G)$ to be the Banach space of all
 higson functions $f$ such that $\nor{\del f}\in \ell^p(G)$.
Note that the definition of $\mathscr H^p(G)$ for all $1\leq p\leq \infty$ does not depend on the choice of finite generating set.

Our interest in higson functions stems from the following ``rigidity'' phenomenon which can be observed under the assumption of ergodicity.
\begin{prop}\label{thm:h-rigid} If $f\in\mathscr H^\infty(G)$ is a function such that \[\int f(g^\inv) d\mu_n(g)\to 0\] for all Reiter sequences $(\mu_n)$, then $f\in C_0(G)$.
\end{prop}

\begin{proof} Suppose by contradiction that $f$ does not belong to $C_0(G)$. Without loss of generality, we would have that there would exist $c>0$ and a sequence $(g_n)$ of elements in $G$ such that $f(g_n)\geq c$ for all $n\in\bb N$. Since $f\in \mathscr H^\infty(G)$, for any finite subset $F\subset G$ there exists $n\in \bb N$ sufficiently large so that $f(h)\geq c/2$ for all $h\in g_n F$. Hence, passing to a subsequence of $(g_n)$, there is a F\o lner sequence $(F_k)$ with the property that $f(h)\geq c/2$ for all $h\in g_{n_k} F_k^\inv$ for all $k\in\bb N$. Taking $\mu_k$ to be the uniform probability measure on the set $F_k g_{n_k}^\inv$, we would then have constructed a Reiter sequence such that $\liminf_k\int f(g^\inv) d\mu_k(g)\geq c/2 >0$, a contradiction.
\end{proof}

\begin{definition}\label{def-growth} Let $G$ be an amenable group, and let $f: G\to \Cal V$ be a large scale lipchitz map. We say that $f$ has \emph{sublinear growth} if $\limsup_{\abs{g}\geq n} \nor{f(g)}/\abs{g}=0$. We say that $f$ has \emph{almost sublinear growth} if $\int \frac{1}{\abs{g}}\nor{f(g)} d\mu_n(g)\to 0$ for all Reiter sequences $(\mu_n)$.
\end{definition}

\begin{prop}\label{thm:sub-compact} Let $G$ be an amenable group. Let $f: G\to \Cal H$ be a  large scale lipschitz map in to Hilbert space. If $f$ is \emph{symmetric}, i.e., $\nor{f(g)}\equiv \nor{f(g^\inv)}$, then the following statements are equivalent:
\begin{enumerate} 
\item $f$ has sublinear growth;
\item $f$ has almost sublinear growth;
\item $f_\xi(g) := \ip{f(g)}{\xi}$ has sublinear growth for all $\xi\in \Cal H$ and the set $V := \Bigl\{\frac{1}{\abs{g}} f(g)\Bigr\}_{g\in G}$ is precompact;
\item $f_\xi$ has almost sublinear growth for all $\xi\in \Cal H$ and $V$ is precompact.
\end{enumerate}
\end{prop}

\begin{proof} The implications (1)$\Rightarrow$(2), (1)$\Rightarrow$(3), and (3)$\Rightarrow$(4) are trivial, while the implication (2)$\Rightarrow$(1) follows directly by Proposition \ref{thm:h-rigid} applied to the function $\frac{1}{\abs{g}}\nor{f(g)}$. Therefore, we only need prove the implication (4)$\Rightarrow$(1). 

To this end, note that if $V$ is precompact, then for any $\e>0$ we can find a set of vectors $\xi_1,\dotsc,\xi_n\in \Cal H$ so that \begin{equation}\int \bigl(\frac{1}{\abs{g}} \nor{f(g)}\bigr)^2 d\mu(g)\leq C \sum_{i=1}^n \int \frac{1}{\abs{g}}\abs{\ip{f(g)}{\xi_i}} d\mu(g) + \e\end{equation} holds for any probability measure $\mu$, where $C := \sup_{g\in G}\frac{\nor{f(g)}}{\abs{g}}<\infty$. Thus, by almost sublinear growth of each $f_\xi$ and the Cauchy--Schwarz inequality, we have that $\int \frac{1}{\abs{g}}\nor{f(g)} d\mu_n(g)\to 0$ along any Reiter sequence. By symmetry, the result then obtains by Proposition \ref{thm:h-rigid}.
\end{proof}

\section{Main Results}

\subsection{Arrays and the weak mean ergodic theorem}

\label{sec:arrays} In this section we present the proof of Theorem \ref%
{main:norm}. Though the theorem was stated explicitly for cocycles, the
natural context for the theorem is actually the class of arrays. This is
essentially due to the fact that there is no well-defined product of
cocycles, while such a product exists for the class of arrays. This allows
us to exploit the weak mixingness in order to derive the strong form of the
theorem in that case.

\begin{theorem}[Theorem \protect\ref{main:norm}]
\label{thm:arraymet} Let $\pi: G\to \mathscr O(\mathcal{H})$ be an
ergodic orthogonal representation of a finitely generated amenable group $G$%
, and let $\alpha: G\to \mathcal{H}$ be an array. Let $S$ be a finite,
symmetric, generating set for $G$, and let $\lvert \,\cdot\,\rvert$ denote
the word length in $S$. If $(\mu_n)_{n\in\mathbb{N}}$ is a Reiter sequence
for $G$, then 
\begin{equation}  \label{wmet}
\int \frac{1}{\lvert g\rvert}\,\alpha(g)\, d\mu_n(g)\to 0
\end{equation}
in the weak topology. If $\pi$ is weakly mixing, then 
\begin{equation}  \label{wmwmet}
\int \frac{1}{\lvert g\rvert}\,\lvert \langle \alpha(g), \xi \rangle\rvert\,
d\mu_n(g)\to 0
\end{equation}
for all $\xi\in \mathcal{H}$.
\end{theorem}

Before we begin the proof, we pause to introduce some convenient notation to
be used here as well as in the sequel.

\begin{notation}
\label{not:array} Let $\alpha:G\to \mathcal{H}$ be an array. We set 
\begin{equation*}
\alpha^\flat(g) = \frac{1}{\lvert g\rvert}\ \alpha(g),
\end{equation*}
where by convention $\alpha^\flat(e) = 0$. $\mathcal{H}\otimes \mathcal{H}$
will be denoted as $\widetilde{\mathcal{H}}$. The representation $%
\pi\otimes\pi: G\to \mathscr O(\widetilde{\mathcal{H}})$ will be denoted as 
$\widetilde\pi$. The array $\widetilde\alpha: G\to \widetilde{\mathcal{H}}$
is defined as 
\begin{equation*}
\widetilde\alpha(g) = \frac{1}{\lvert g\rvert}\ \alpha(g)\otimes\alpha(g),
\end{equation*}
where $\widetilde\alpha(e) = 0$ by convention.
\end{notation}

\begin{proof}[Proof of Theorem \ref{thm:arraymet}]
The proofs of these formulas are inspired by the standard approach to the
(weak) mean ergodic theorem for amenable groups. We
begin by proving (\ref{wmet}). To this end, we fix $\epsilon>0$, $n\in\mathbb{N%
}$ and note that there exists a finite subset $F_n\subset G$ such that 
\begin{equation*}
\nor{\al^\f(gh) - \pi(g)\al^\f(h)}\leq \epsilon
\end{equation*}
whenever $g\in B(n)$ and $h\in G\setminus F_n$. Let $\xi\in \mathcal{H}$ be
a vector of the form $\xi = (1 - \pi(g^{{\scriptscriptstyle {-1}}}))\eta$
for some $g\in B(n)$, $\eta\in \mathcal{H}$. We then have that 
\begin{equation}  \label{eq:almost}
\begin{split}
&\bbs{\int \ip{\al^\f(h)}{\xi}\, d\mu_N(h)} \\
&= \bbs{\int \ip{\al^\f(h) - \pi(g)\al^\f(h)}{\eta}\, d\mu_N(h)} \\
&\leq \bbs{\int \ip{\al^\f(gh) - \pi(g)\al^\f(h)}{\eta}\, d\mu_N(h)} + \int
\lvert \langle \alpha^\flat(k), \eta \rangle\rvert \,d\lvert \mu_N(g^{{%
\scriptscriptstyle {-1}}}k) - \mu_N(k)\rvert \\
&\leq \nor{\eta} \int \nor{\al^\f(gh) - \pi(g)\al^\f(h)}\, d\mu_N(h) + \sup_k%
\nor{\al^\f(k)}\cdot\nor{\mu_N - g\ast\mu_N}_1\lesssim 2\nor{\eta}\epsilon,
\end{split}%
\end{equation}
since $\lim_N \mu_N(F_n)= 0$ and $\nor{\al^\f}$ is bounded. By inspection,
the estimate holds for the span $\mathcal{V }:= span\{ \xi : \exists
g\in G,\eta\in \mathcal{H}(\xi = (1-\pi(g))\eta)\}$, establishing the
theorem in that case. Since $\int \nor{\al^\f(g)}\, d\upsilon_n(g)$ is
uniformly bounded, the result then extends to the closure of $\mathcal{V}$,
which by ergodicity is all of $\mathcal{H}$. This concludes the proof of %
(\ref{wmet}).

For the proof of the second part, formula (\ref{wmwmet}), we note that if $%
\alpha: G\to \mathcal{H}$ is an array for $\pi$, then $\widetilde\alpha(g)$
is an array for $\widetilde\pi$ by Lemma \ref{lem:tensor}. Applying this, we
see that 
\begin{equation}
\bbs{\int \ip{\widetilde\al^\f(h)}{\xi\otimes\xi}\, d\mu_N(h)} = \int \lvert
\langle \alpha^\flat(h), \xi \rangle\rvert^2\, d\mu_N(h)\to 0
\end{equation}
by the proof of (\ref{wmet}). By the Cauchy--Schwarz inequality, we have that 
\begin{equation}
\int \lvert \langle \alpha^\flat(h), \xi \rangle\rvert\, d\mu_N(h)\leq \Bigl(%
\int \lvert \langle \alpha^\flat(h), \xi \rangle\rvert^2\, d\mu_N(h)\Bigr)%
^{1/2},
\end{equation}
and we are done.
\end{proof}

In the case the $1$-cocycle is proper, there is a sharpening of the above result. The proof is identical the the proof of the previous theorem, using Proposition 1.4 from \cite{CS} instead of Lemma \ref{lem:tensor}.

\begin{prop} Let $\pi: G\to \Cal H$ be a weakly mixing orthogonal representation. If $b: G\to \Cal H$ is a proper $1$-cocycle, then \begin{equation} \int \frac{1}{\nor{b(g)}} \abs{\ip{b(g)}{\xi}} d\mu_n(g)\to 0 \end{equation} for all Reiter sequences $(\mu_n)$.
\end{prop}

\subsection{Theorem \protect\ref{main:controlled} and the mean ergodic theorem}

We begin with the main technical theorem in this section, the formulation and proof of which are inspired by Lemma 3.4 in \cite{CTV}.

\begin{theorem}\label{thm:w-sublinear}
Let $G$ be a finitely generated, discrete group in the class $\Cal{CF}$. Let $b:G\to \mathcal{H}$ be a $1$-cocycle associated to an orthogonal representation $\pi$.
Assume that \begin{equation}\label{eq:c0} \frac{1}{\abs{g}}\ip{b(g)}{\xi}\in C_0(G)\end{equation} for all $\xi\in \Cal H$ (i.e., $b$ is weakly sublinear). Let $(F_n)_{n\in\mathbb{N}}$
be a $K$-controlled F\o lner sequence. Let $\upsilon_n$ be the uniform
measure on $F(n)$. There exists a sequence $(\mu_{k})$ of finitely supported
measures $\mu_k\in co\{\upsilon_n : n\in \mathbb{N}\}$ such that $\xi_k
:= \int b(g) \,d\mu_{k}(g)$ form a sequence of almost fixed points for
the affine action $G\curvearrowright^T \mathcal{H}$ associated to $b$.
\end{theorem}

\begin{proof}
Fix a word length $\lvert \,\cdot\,\rvert$ coming from some finite,
symmetric generating set $S\subset G$. Let $d_n = \diam F_n$. We set $%
F_n(g) = gF_n\Delta F_n\subset \partial F_n\subset B(d_n + 1)$, for each $%
g\in S$. Let $\eta_n = \int b(g) \,d\upsilon_n(g)$.

For all $n\in \mathbb{N}$ we have the \emph{a priori} estimate 
\begin{equation}  \label{eq:apriori}
\begin{split}
\nor{T_g(\eta_n) - \eta_n} &= \bnor{\int b(h) \,d\up_n(g^{\inv}h) -\int b(h)
\,d\up_n(h)} \\
&\leq \frac{1}{\lvert F_n\rvert}\int_{F_n(g)} \nor{b(h)}\, dh \\
&\leq C (d_n+1)\cdot\frac{\lvert \partial F_n\rvert}{\lvert F_n\rvert}\leq
2 C K,
\end{split}%
\end{equation}
where $C = \sup_{s\in S} \nor{b(s)}$.

Therefore, we need only show that for any $\xi\in \mathcal{H}$ and $g\in S$,
we have that 
\begin{equation}
\lim_n\lvert \langle T_g(\eta_n) - \eta_n, \xi \rangle\rvert= 0.
\end{equation}
Indeed, the sequence $(T_g(\eta_n) - \eta_n)_{n\in\mathbb{N}}$ would then
have $0$ as a weak limit point for any $g\in S$. Thus, the sequence $%
\bigoplus_{g\in S} (T_g(\eta_n) - \eta_n) \subset \bigoplus_{g\in S} 
\mathcal{H}$ converges weakly to $0$, so that by passing to the convex hull,
the theorem obtains.

We now fix $\xi\in\mathcal{H}$. By assumption \ref{eq:c0} for every $%
\epsilon>0$ there exists a finite set $E_\epsilon\subset G$ such
that 
\begin{equation}
\lvert \langle b(g), \xi \rangle\rvert<\epsilon\lvert g\rvert
\end{equation}
for all $g\in G\setminus E_\epsilon$.

Since $\lim_n\upsilon_n(E_\epsilon)=0$, we have that for any $g\in
S$,  
\begin{equation}
\begin{split}
\lvert \langle T_g(\eta_n) - \eta_n, \xi \rangle\rvert &= \bbs{\int
\ip{b(h)}{\xi} \,d\up_n(g^{\inv}h) -\int \ip{b(h)}{\xi} \,d\up_n(h)} \\
&= \frac{1}{\lvert F_n\rvert}\bbs{\int_{F_n(g)} \ip{b(h)}{\xi} \,dh} \\
&\leq \frac{1}{\lvert F_n\rvert}\int_{F_n(g)} \lvert \langle b(h), \xi
\rangle\rvert\,dh \\
&\lesssim 2\epsilon (d_n+1)\frac{\lvert \partial F_n\rvert}{\lvert
F_n\rvert}\leq 4K\epsilon,
\end{split}%
\end{equation}
and we are done.
\end{proof}

Hence, if $G$ is a group in the class $\Cal{CF}$, then for any $1$-cocycle $b: G\to \Cal H$ associated to some orthogonal representation, $b$ is not almost inner only if there exists a vector $\xi\in \Cal H$ so that $f(g) := \ip{b(g)}{\xi}$ does not have sublinear growth.

\begin{question} By Proposition 3.1 in \cite{CTV} we know that any almost inner $1$-cocycle has sublinear growth. For a general amenable group, is it the case that any weakly sublinear $1$-cocycle is in fact (strongly) sublinear?
\end{question}

\begin{theorem}[Theorem \ref{main:controlled}] Let $G$ be finitely generated group in the class $\Cal{CF}$. Let $\pi: G\to\mathscr O(\Cal H)$ be an orthogonal representation, and let $b: G\to \Cal H$ be a 1-cocycle associated to $\pi$. Suppose that \begin{equation}\int \frac{1}{\abs{g}}\ip{b(g^{\inv})}{\xi} d\mu_n(g)\to 0\end{equation} for all $\xi\in \Cal H$ and all Reiter sequences $(\mu_n)$. Then the affine action
$G\curvearrowright^T \mathcal{H}$ associated to $b$ admits a
sequence of almost fixed points.
\end{theorem}

\begin{proof} The proof follows directly from Proposition \ref{thm:h-rigid} combined with Theorem \ref{thm:w-sublinear}.
\end{proof}

\begin{definition}\label{def-harmonic} Let $G$ be a finitely generated group and let $\mu$ be a probability measure on $G$. A function $u: G\to \Cal V$ into a vector space is said to be \emph{$\mu$-harmonic} if \begin{equation} u(g) = \int u(gs) d\mu(s)\end{equation} for all $g\in G$.
\end{definition}

Let $\mu$ be a probability measure with finite second moment, i.e., $\int \abs{g}^2 d\mu(g)<\infty$. We know, cf.\ Theorem 6.1 in \cite{ShalRC} and Theorem 6.1 in \cite{CreuThesis}, that every group $G$ without property (T) of Kazhdan admits at least one $\mu$-harmonic $1$-cocycle for some (irreducible) representation.

\begin{prop} Let $G$ be a group in the class $\Cal {CF}$, $\pi: G\to \mathscr O(\Cal H)$ be an orthogonal representation, and $b: G\to \Cal H$ be a $\mu$-harmonic $1$-cocycle with $\mu$ having finite second moment. Let $\pi_0$ be the restriction of $\pi$ to the (invariant) subspace $\Cal H_0$ spanned by the image of $b$. If $V := \Bigl\{\frac{1}{\abs{g}}b(g)\Bigr\}$ is precompact, then $\pi_0$ is compact.
\end{prop}

\begin{proof} Suppose by contradiction that $\Cal H_0$ contains an non-zero invariant subspace $\Cal K$ on which the restriction of $\pi$ is weakly mixing. Setting $b': G\to \Cal K$ defined by $b'(g) := P_{\Cal K} b(g)$, we then would have that $b'$ is a harmonic $1$-cocycle into a weakly mixing representation such that $V' := P_{\Cal K} V = \Bigl\{\frac{1}{\abs{g}} b'(g)\Bigr\}$ is precompact. Proposition \ref{thm:sub-compact} then implies that $b'(g)$ has sublinear growth; hence, by Theorem \ref{thm:w-sublinear} it is almost inner. However, no non-zero harmonic $1$-cocycle into an orthogonal representation can be almost inner, cf. Theorem 6.1 in \cite{CreuThesis}. Therefore, $b'\equiv 0$ which contradicts the fact that the span of $V'$ is dense in $\Cal K$. Thus, we have shown that $\pi_0$ contains no non-zero, weakly mixing subrepresentation which implies that $\pi_0$ is compact.
\end{proof}

\section{Final Remarks and Open Problems}\label{Final-Remarks}

\subsection{On the growth of harmonic functions}

\begin{Conjecture}\label{mconj:harmonic} Let $G$ be an amenable group, and let $\mu$ be a probability measure with finite second moment and trivial Poisson boundary (cf.\ \cite{KaiVer}). If $u: G\to \bb R$ is a lipschitz $\mu$-harmonic function such that \[\int \frac{1}{\abs{g}} \abs{u(g)} d\mu_n(g)\to 0\] for all Reiter sequences $(\mu_n)$, then $u$ has sublinear growth.
\end{Conjecture}

\noindent Notice that if $u$ is harmonic, then $\abs{u}$ is \emph{subharmonic}, i.e., $\abs{u}(g)\leq \frac{1}{\abs{S}}\sum_{s\in S} \abs{u}(gs)$ for all $g\in G$, so the conjecture may be posed in this generality.

\begin{definition}\label{h-fd} A group $G$ has \emph{property \HFD} of Shalom if any affine action $G\ca^T \Cal H$ on Hilbert space with weakly mixing linear part admits almost fixed points.
\end{definition}

\begin{prop} Suppose that Conjecture \ref{mconj:harmonic} holds for a group $G$ which admits a controlled F\o lner sequence. Then $G$ has property \HFD.
\end{prop}

\noindent The proof is an easy consequence of Theorem \ref{thm:w-sublinear}.

\begin{Remark} A result of Hebisch and Saloff-Coste, Theorem 6.1 in \cite{HebSC}, shows that there a no non-constant real-valued harmonic functions of sublinear growth on a group of polynomial growth. It would be interesting if a variant of this argument could be made to apply to harmonic functions of almost sublinear growth.
\end{Remark}

A stated in Proposition \ref{prop-cf}, the known classes of amenable groups which admit controlled F\o lner
sequences are: groups of (weak) polynomial growth; polycyclic groups, i.e.,
lattices in solvable Lie groups; wreath products $D \wr \mathbb{Z}$ with $D$
finite; semi-direct products $\mathbb{Z}[\frac{1}{mn}]\rtimes_{m/n} \mathbb{Z%
}$, with $m,n$ coprime and $\lvert mn\rvert\geq 2$. The latter three classes
are the work of Tessera, Theorem 11 in \cite{Tessera2}. Each of these
classes is known have property H$_{\text{{\tiny FD}}}$ by the seminal work
of Shalom, Theorems 1.13 and 1.14 in \cite{Shalom}. The advantage to the
approach suggested here is that it may potentially offer a broad, conceptually unified way of
deriving property H$_{\text{{\tiny FD}}}$.  

It follows from an argument given in \cite{Shalom} (Theorem 6.7.2) that the conjecture implies Gromov's theorem. The approach presented here is conceptually
quite different from other approaches to Gromov's theorem \cite{BGT, Gromov,
Hrushovski, Kleiner, ShalomTao}.

We also point out that another consequence of Conjecture \ref{mconj:harmonic} would be that
there are solvable groups, e.g., $\mathbb{Z }\wr \mathbb{Z}$, which do not
admit controlled F\o lner sequences, cf.\ Theorem 1.15 in \cite{Shalom}.

\subsection{On the space $\mathscr H^p(G)$} As a last remark, we develop another line of thought towards establishing the mean ergodic theorem for affine actions of groups of polynomial growth independently of Gromov's theorem.

\begin{theorem}\label{lem:h1} Let $G$ be a one-ended group with a finite, symmetric, generating set $S$. If $f\in \mathscr H^1(G)$, then $f\in C_0(G) + \bb C 1$.
\end{theorem}

\begin{proof} For every $\e>0$, choose $r$ sufficiently large so that \[K_r := \sum_{g\in G\setminus B_r}\sum_{ s\in S} \abs{f(g) - f(gs)}<\e.\] Since $G$ is one-ended $G\setminus B_r$ contains exactly one infinite connected component $U_r$. For every pair of elements $g,h\in U_r$ there exists a sequence of elements $x_1,\dotsc,x_n$ in $ U_r$ such that $g = x_1$, $h= x_n$ and $x_{i+1}^\inv x_i\in S$ for all $i=1,\dotsc,n-1$. Hence it follows by the triangle inequality that \[\abs{f(g) - f(h)}\leq K_r\] which proves the claim.
\end{proof}

In fact, in the case that $f$ is positive, a slightly weaker condition will suffice:

\begin{theorem} For $f\in \ell^\infty(G)$ and $F\in \ell^\infty(G\times S)$, let $f\cdot F(g,s) := f(g)F(g,s)$. Let $G$ be a one-ended group with a finite, symmetric, generating set $S$. Suppose that $f\in \ell^\infty(G)$, $f\geq 0$. If $\nor{f\cdot \del f}\in \ell^1(G)$, then $f\in C_0(G)+ \bb C1$.
\end{theorem}

\noindent Note that since $f\geq 0$, we have that $\nor{f\cdot \del f}\leq \nor{\del(f^2)}$; hence, by the boundedness of $f$ and standard estimation techniques it follows that if $f^p\in \mathscr H^1(G)$ for any $1\leq p<\infty$, then it holds that $f\in C_0(G) + \bb C1$.

\begin{proof} Let $\Gamma = \Gamma(G,S)$ be the Cayley graph of $G$ with respect to the generating set $S$. We produce a new graph $\Gamma'$ by subdividing each edge in $\Gamma$ so the the vertex set of $\Gamma'$ may be identified with $V(\Gamma)\sqcup E(\Gamma)$ and $\Gamma'$ is again one-ended. We define a map $f': V(\Gamma')\to \bb R$ by $f'(g) := f(g)^2$ for $g\in V(\Gamma)$ and $f'(e) := f(g)f(gs)$ for $e = (g,gs)\in E(\Gamma)$. Now by assumptions we can see that $\nor{\del f'}\in \ell^1( V(\Gamma'))$, so by Theorem \ref{lem:h1}, we can conclude that $f^2\in C_0(G) + \bb C1$. By the positivity of $f$, this suffices to show the result.
\end{proof}

\begin{prop} Let $G$ be a one-ended group in the class $\Cal{CF}$. If $b$ is a $1$-cocycle associated to an ergodic representation $\pi: G\to \mathscr O(\Cal H)$ such that \begin{equation}\frac{1}{\abs{g}} \ip{b(g)}{\xi}\in \mathscr H^1(G)\end{equation} for all $\xi\in \Cal H$, then $b$ is almost inner. The same holds assuming that $\pi$ is weakly mixing and \begin{equation}\frac{1}{\abs{g}} \abs{\ip{b(g)}{\xi}}\in \mathscr H^1(G).\end{equation}\end{prop}

\begin{proof} The proof follows directly from Theorem \ref{lem:h1} and Theorem \ref{thm:w-sublinear}.
\end{proof}

\begin{prop} If $G$ is a group of polynomial growth, then there exists $1\leq p<\infty$ such that for any $1$-cocycle $b: G\to \Cal H$ we have that  \begin{equation}\frac{1}{\abs{g}} \ip{b(g)}{\xi}\in \mathscr H^p(G)\end{equation} for all $\xi\in \Cal H$
\end{prop}

\begin{proof} Fixing a finite generating set $S$, we have that $\sum_{s\in S} \nor{\frac{1}{\abs{g}} b(g) - \frac{1}{\abs{gs}} b(gs)}\ll \frac{1}{\abs{g}}$ choosing an integer $p$ such that $R^{p-2}\gg \abs{B(R)}$, we have that \begin{equation}\sum_{g\in G}\sum_{s\in S} \bnor{\frac{1}{\abs{g}} b(g) - \frac{1}{\abs{gs}} b(gs)}^p\ll \sum_{g\in G} \abs{g}^{-p}\ll \sum_{n\in \bb N} n^{-2}\end{equation} from which the result easily obtains.\end{proof}

We formulate the following conjecture as an alternative approach to Gromov's theorem.

\begin{Conjecture}\label{thm:hp+} If $G$ is a one-ended group of polynomial growth, then for any $1\leq p<\infty$ any \emph{positive} function  $f\in \mathscr H^p(G)$ belongs to $C_0(G) + \bb C1$.
\end{Conjecture}

\section*{Acknowledgements}

We would like to thank Professors Yehuda Shalom and Terence Tao for teaching
seminars on the various approaches to Gromov's theorem on groups of
polynomial growth in the Fall 2011 quarter at UCLA which stimulated our
thoughts in this direction. We are grateful to Darren Creutz and Jesse Peterson for useful comments.
 We are especially grateful to Yehuda Shalom for encouragement.

\end{document}